\documentclass[a4paper,10pt]{article}

\newcommand\nc{\newcommand}

\nc{\od}[2]{\frac{d#1}{d#2}}
\nc{\be}{\begin{equation}}
\nc{\ee}{\end{equation}}
\nc{\bd}{\begin{displaymath}}
\nc{\ed}{\end{displaymath}}
\nc{\bq}{\begin{eqnarray}}
\nc{\eq}{\end{eqnarray}}
\nc{\p}{\partial}
\nc{\ra}{\rightarrow}
\nc{\R}{\mathbb{R}}
\nc{\dx}{\mathrm{d}x}
\nc{\curl}{\operatorname{curl}}
\nc{\dv}{\operatorname{div}}
\nc{\meas}{\operatorname{meas}}
\nc{\supp}{\operatorname{supp}}
\nc{\medotimes}{\resizebox{11pt}{!}{$\otimes$}}
\nc{\T}{\mathrm{T}}
\nc{\wk}{\rightharpoonup}
\nc{\ep}{\varepsilon}
\nc{\hep}{\hat{\ep}} 
\nc{\Id}{\mathrm{Id}}

\nc{\comment}[1]{{\tt \color{red}#1}}

\usepackage{color}
\usepackage{graphics}
\usepackage{amssymb, amsmath, amsthm}

\usepackage[utf8]{inputenc}

\newtheorem{proposition}{Proposition}[section]

\newtheorem{theorem}[proposition]{Theorem}
\theoremstyle{remark}
\newtheorem{remark}[proposition]{Remark}
\theoremstyle{definition}

\title{Relaxation of the single-slip condition in strain-gradient plasticity}
\author{Keith Anguige and Patrick W.~Dondl}

\begin{document}

\maketitle

\begin{abstract}
We consider the variational formulation of both geometrically linear and geometrically nonlinear elasto-plasticity subject to a class of hard single-slip conditions. Such side conditions typically render the associated boundary-value problems non-convex. We show that, for a large class of non-smooth plastic distortions, a given single-slip condition (specification of Burgers vectors) can be relaxed by introducing a microstructure through a two-stage process of mollification and lamination. The relaxed model can be thought of as an aid to simulating macroscopic plastic behaviour without the need to resolve arbitrarily fine spatial scales.
\end{abstract}

\section{Introduction}\label{sec:1}

Complex subgrain dislocation patterns in plastically deformed metallic crystals have been observed by many authors \cite{han,hug,hug2,jin,nix}, and these dislocation microstructures are known to affect greatly several aspects of plastic behaviour, such as work hardening, and the Bauschinger and Hall-Petch effects. Currently, attempts at modelling such behaviour are often phenomenological, which leads to significant problems for specimens on the micron scale \cite{uchic,dem}. Nevertheless, predictive models for pattern formation \emph{have} been proposed, for example in the seminal work by Ortiz and Repetto \cite{ortiz}. Such models are based on a non-quasiconvex energy minimization at each time increment of the evolution, the non-convexity resulting from the assumption of latent hardening in single-crystal plasticity, which means that if dislocations from different slip systems meet, then they form energetically favorable (but sessile) atomistic reaction products, such as the so-called Lomer-Cottrell locks. 
This increases the 
dissipation potential for any plastic 
deformation that does not occur in single slip.

Here, we proceed as in \cite{ortiz}, and in addition make the simplifying assumption of infinite latent hardening (i.e. that the material should deform in single slip at each point), thus resulting in a strain-gradient-plasticity model. This is akin to the approach adopted in, say, \cite{dolz,dolz2}, amongst many other works, or, to take another example, \cite{dm}, where it was shown that, for a simple shear experiment, the predictions made by non-convex single-slip models can indeed be observed in plastically deformed single crystals. Despite our emphasis on the single-slip condition here, however, it should be noted that the our main result, namely that such non-convexities should always be partially relaxed, still holds for finite latent hardening between slip systems in a single slip-plane, at least as long as this is implemented as a direction-dependent dissipation (see Remark~\ref{rem:non-infinite}). 

In incremental models of non-convex crystal plasticity, a strain-gradient penalty term of type $\int |\curl \beta|$, where $\beta$ is the plastic strain, is commonly added \cite{co}, and is actually necessary in some cases if one wants to make the hardening physically realistic~\cite{Fokoua_14}. In rectangular shear samples of small size, the strain-gradient energy and single-slip condition together play a dramatic role, as we showed analytically in \cite{ang}. Specifically, for a strain-gradient model of a single crystal with B2 symmetry, there are three qualitatively different energy-scaling regimes, which are determined by the  aspect ratio, L : when L is small the energy scales quadratically (i.e. elastically) with the shear, when L is large one sees easy shear-band formation, and for intermediate values of L there is a regime in which the energy scales linearly with the applied shear. One consequence of this is that a micron-sized sample consisting of only a few grains will easily shear off if a slip 
system under 
stress 
connects free surfaces \cite{uchic}. 

In this article, we show that such a strain-gradient penalty term is \emph{not} sufficient to regularize the problem in the sense of making the associated elasto-plastic energy lower semicontinuous. In some sense, this problem was already discovered in numerical simulations, for example by Hildebrand \emph{et al.} \cite{hild}. These authors circumvent the problem by introducing an ultra-fine full-gradient-penalty regularisation -- this, however, introduces an even finer length scale which is very difficult to treat numerically in an industrial context. A justification for models that include a strain-gradient-type penalisation of geometrically necessary dislocations (GNDs), and also potentially admit a non-convexity in the dissipation (implemented through a hardening matrix in the flow rule) can be found in~\cite{Acharya_2000}, while a recent solution-method in the framework of finite-element modeling is described in~\cite{Ma_2006}. We believe that some improvement to the numerical stability of these models 
is possible if the non-convexity is \emph{relaxed} according to our prescription, which we now describe.

Continuing in the spirit of \cite{ang}, and motivated by the remarks above, in this paper we will thus consider the problem of relaxing the geometrically linear energy where
\be
\label{eq:main_energy}
E(u,\beta) =
\left\{
\begin{array}{ccl}
\int_{\Omega} | (\nabla u - \beta)_\mathrm{sym}|^2\dx + 
\sigma\|\curl_{\mathcal{M}}\beta\|_{\Omega} + \tau\int_{\Omega} |\beta|\dx & : &\textbf{(SSC)}~\textrm{holds}\\
& & \\
+\infty &:&\textrm{otherwise} 
\end{array}
\right.,
\ee
subject to a family of Dirichlet conditions at $\Gamma_D\subset\p\Omega$ on the vector-valued displacement $u~\colon \Omega\to \R^3$, where $\mathcal{M} = \{m_j: j=1,\ldots, M\}$ is a family of slip normals, and (SSC) represents a class of \emph{single-slip} conditions (to be defined shortly) on the matrix-valued plastic deformations, $\beta = s\otimes m\colon \Omega\to \R^{3\times 3}$. Here, the subscript `$\mathrm{sym}$' denotes taking the symmetric part of the matrix in parenthesis. The first term in the integrand of~\eqref{eq:main_energy} is the linearised elastic energy of the specimen, and the second term penalises GNDs, which are quantified by the row-wise curl of the Nye tensor multiplied by the line-tension parameter $\sigma\geq 0$, taken \emph{individually} for each slip plane $m_j^{\perp}$. To be more explicit, assuming that $\Omega_j$ is the set where $\beta$ can be written as $\beta = s\otimes m_j$, with $m_j$ pointing in the $x^1$-direction, and $s\in m_j^{\perp}$, we define
\be
\|\curl_{m_j}\beta\|_{\Omega} =  \int_{\Omega} |\curl_{m_j}\beta|~\mathrm{d}x^1,\label{curl}
\ee
where the {\em planar curl density}, $|\curl_{m_j}\beta|$, is given by
\be\label{TV}
|\curl_{m_j}\beta|(\Omega\cap\{x^1=t\}) = \sup_{
\substack{\phi\in\left(C^1(\Omega\cap\{x^1=t\})\right)^{2\times 2}\\
|\phi|\leq 1}}
\sum_{i = 2,3}\iint_{\Omega\cap\{x^1=t\}}\chi_{\Omega_j}s_i(y,z)(\dv\phi)_i~\mathrm{d}y\mathrm{d}z,
\ee
and where we have used (and often {\em will} use) $y$ and $z$ instead of $x^2$ and $x^3$ for brevity. With this in hand, we define
\be
\|\curl_{\mathcal{M}}\beta\|_{\Omega} = \sum_{j=1}^M \|\curl_{m_j}\beta\|_{\Omega}.
\ee

Thus, for $\beta\in L^1$ or $L^2$, say, $\|\curl_{\mathcal{M}}\beta\|$ is finite iff $\beta\in BV$ on almost every slip plane \emph{and} the planar total-variation norms are integrable in the normal direction. For convenience, we will sometimes use the total-variation notation $|\curl\beta| = |D_{y,z}\beta|$.

Note that our rather cumbersome definition of the curl was necessary in order to avoid the introduction of spurious dislocation-cancellation effects. On the other hand, we \emph{are} making the assumption that colliding dislocations from different slip planes do \emph{not} cancel -- for a discussion of this assumption in a simplified scalar model, see~\cite[Chapter 4]{co}. In the two dimensional model used in~\cite{ang}, such cancellations naturally do not occur. 

The third term in (\ref{eq:main_energy}) represents a rate-independent dissipation, such that $\tau\geq 0$ measures the (isotropic) critical resolved shear stress. One could also consider, instead, a more general hardening term of the form $\int |\beta|^{\alpha}$ with $0<\alpha\le 2$, and our results can be generalised, in a straightforward manner, to include this. For the special case $\alpha<1$, the dissipation relaxes to zero, leading to slip-band formation and the treatment of the singularities that arise here (see Remark \ref{rem:non-infinite}). For a recent discussion of such sublinear dissipation in the context of strain-gradient plasticity, we refer to~\cite{Fokoua_14}.

Relaxation of the energy functional (\ref{eq:main_energy}) was only briefly alluded to in \cite{ang}, in the course of proving an energy upper bound for a shear experiment on a B2 crystal. Here we wish to derive the relaxed model rigorously for a class of slip systems which includes the case of B2 symmetry. On general theoretical grounds, one expects the relaxed model to reproduce the correct macroscopic plastic behaviour of the single crystal, and hence to facilitate efficient numerical simulation of this behaviour without the need to resolve arbitrarily fine spatial scales.

The slip conditions considered in this article are as follows. We will assume there are $M$ slip planes available, and, as already mentioned, the set of possible slip-plane normals, $m_i$, $i=1,2,\ldots,M$, will be denoted by $\mathcal{M}$. In each $m_i^{\perp}$ there will be two arbitrary \emph{normalised} Burgers vectors, $b_{j,i}$, $j=1,2$, and the single-slip (resp. relaxed-slip) conditions on $\beta=s\otimes m$ read
\begin{description}
\item[(SSC)] For a.e. $x = (x^1,x^2,x^3)\in \Omega$, there holds~$m = m_{i(x)}\in\mathcal{M}$ and $s(x)\in\mathrm{Sp}_{j = 1,2}\{b_{j,i(x)}\}$.   
\item[(RSC)] For a.e. $x = (x^1,x^2,x^3)\in \Omega$, there holds~$m = m_{i(x)}\in\mathcal{M}$ and $s(x)\in m^{\perp}_{i(x)}$,
\end{description}
such that \textbf{(SSC)} includes the B2 lattices that were considered in~\cite{ang}. The conditions also allow the results to be applied to the most important cases, namely those of fcc and bcc crystal lattices, in the sense that our upper bound for the relaxed energy is perfectly valid. In such materials there are, of course, more than two Burgers vectors in each slip plane. It is then potentially possible to reduce the curl-term further, by choosing an optimal decomposition of the relaxed slip into single slip. Such an additional minimisation can easily be included in a numerical simulation.

Our main result can now be described heuristically as follows: in any elasto-plastic model of type~\ref{eq:main_energy}, one should, for the purposes of numerical simulation, use the relaxed-slip condition (RSC) instead of the single-slip condition (SSC). As can be seen in Proposition~\ref{prop:rel_energy}, this also requires a slight modification of the strain-gradient term and the dissipation, but the new terms can be explicitly computed from the slip-system configuration and the original strain: Theorems~\ref{thm:main} and~\ref{thm:nonlinear} constitute a precise statement of all this.

The relaxation result is proved rigorously in two stages. First, starting with a sufficiently smooth, relaxed pair $(u,\beta)$, we derive a relaxed energy from (\ref{eq:main_energy}) by laminating between the two Burgers vectors, $b_{j,i}$, on each slip patch (the second index will be dropped when working on a fixed slip patch, $m_i = \mathrm{const.}$). Next, we note that the relaxed energy also makes sense for a larger class of non-smooth relaxed slips, and the second step of our procedure is to show that such slips can be mollified without increasing the relaxed energy. This is highly nontrivial, since applying a standard Friedrichs mollifier to a non-smooth $\beta$ can violate the relaxed side condition near slip-patch boundaries.

The paper is organised as follows. In Section \ref{sec:2} we show how to laminate a smooth, relaxed plastic strain, and we derive the relaxed energy. Section \ref{sec:3} is concerned with mollifying non-smooth slips while preserving \textbf{(RSC)}, and in Section \ref{sec:5} we state a summary of our results in the geometrically linear case. Finally, in Section \ref{sec:4} we give a brief treatment of the corresponding results in the geometrically nonlinear case, where the elasto-plastic decomposition of the deformation is multiplicative, rather than additive.

\section{Lamination, relaxed energy}\label{sec:2} 

Suppose we have a displacement $u\in H^1$ on $\Omega$ satisfying some Dirichlet condition, and a~ $C^\infty$, relaxed plastic distortion $\beta = s\otimes m$, $s\in m^{\perp}$, $m\in\mathcal{M}$, which has compact support in a Lipschitz domain $\Omega_1\subset\subset\Omega$, and denote by $c_j$, $j=1,2$, the components of $s$ with respect to the decomposition along normalised Burgers vectors $b_j$, so that 
\be
s = \sum_{j=1}^2 c_j b_j,\label{c}
\ee
and $c_j\in C^{\infty}$.

Suppose also that we choose Cartesian coordinates $(x^1,x^2,x^3) = (x,y,z)$ in $\Omega_1$ so that the orthonormal basis vectors, $e_i$, are arranged with $e_1 = m$. 

\subsection{Laminating $\beta$}
Now laminate $\beta$ by filling $\Omega_1$ with a stack of bi-layers, each parallel to $m^{\perp}$ and having thickness $\frac{1}{2^n}$, $n\in\mathbb{N}$, and then defining on each successive bi-layer an alternating (in the $x^1$-direction), unrelaxed plastic distortion, $\beta_n$, by
\be
\beta_n =
 \left\{
\begin{array}{lcl}
2c_1 b_1\otimes m & : & \textrm{top slice}\\
2c_2 b_2\otimes m & : & \textrm{bottom slice}\\
\end{array}
\right.,\label{slice}
\ee  
where the $c_j$ are evaluated on the centre-plane of bi-layer in (\ref{slice}), all slices have the same thickness, and, to be concrete, the $(n+1)$-th laminate is obtained from the $n$-th by bisecting each of the bi-layers along a slip plane. For the moment, we have neglected to describe precisely what happens in the neighbourhood of $\p\Omega_1$ -- see, however, below.

Since $\beta$ is assumed smooth, it is easy to see that $\beta_n\wk\beta\in L^p(\Omega_1)$ for any $p\in[1,\infty)$ as $n\ra\infty$. For example, note that simple functions are dense in $L^{p'}$, the family $\beta_n$ is uniformly bounded in $L^\infty$ and the proof that $\int\beta_n g\ra \int\beta g$ as $n\ra\infty$ for $g$ simple is trivial. In the same way, the analogous result for weak $L^1$ convergence also follows.

Defining $L^1_t= L^1(\Omega_t)$, where $\Omega_t = \Omega_1\cap \{x^1 = t\}$, we also immediately see that the strain-gradient energy satisfies
\be
\int_{\Omega_1} \|D_{y,z}\beta_n\|_{L^1_{t'}}\mathrm{dt}'\ra\sum_{j=1}^2 \int_{\Omega_1}\|D_{y,z}c_j\|_{L^1_{t'}}~\mathrm{d}t',
\label{lam_curl}
\ee
as $n\ra\infty$, whereby the rhs also makes sense for 
\be
c_j : I_{\Omega_1}\ni t\mapsto BV(\Omega_t)~\textrm{with}~\int_{I_{\Omega_1}}|D_{y,z}c_j|~\mathrm{d}t<\infty,~\textrm{such that}~I_{\Omega_1}= \{t\in\mathbb{R} : \Omega_t\neq\emptyset\}.
\ee
We call the integral on the rhs of (\ref{lam_curl}) the {\em laminated curl} of $\beta$, denoted by $\|\curl_m\beta\|_{\mathrm{lam}}$, and the corresponding integrand the {\em laminated curl density}, $|\curl_m\beta|_{\mathrm{lam}}$: thus, $\|\curl_m\beta\|_{\mathrm{lam}} = \int |\curl_m\beta|_{\mathrm{lam}}~\dx^1$.

Meanwhile, for the dissipation term we get
\be
\int_{\Omega_1} |\beta_n|~\dx\ra\sum_{i=1}^2\int_{\Omega_1} |c_i|~\dx\label{hard},
\ee
as $n\ra\infty$, and we call the right-hand integral in (\ref{hard}) the {\em laminated hardening} of $\beta$, denoted by $\|\beta\|_{1-\mathrm{lam}}$.

We have the following inequalities for the laminated curl and the laminated hardening.

\begin{proposition}
If we have a fixed $m_j\in\mathcal{M}$, $s\in BV((\Omega_1)_t,m_j^{\perp})$, some $t\in\mathbb{R}$, then we have, for $\beta = s\otimes m_j$ and $\beta = \sum_i c_i b_i$,
\be
|\curl_{m_j}\beta|\leq (|b_2.b_1| + |b_2.b_1^{\perp}|)|\curl_{m_j}\beta|_{\mathrm{lam}}\leq\sqrt{2}|\curl\beta|_{\mathrm{lam}}.\label{equiv2}
\ee
and
\be
|\curl_{m_j}\beta|_{\mathrm{lam}}\leq\sqrt{2}\left(\frac{1 + |b_1.b_2|}{|b_1^{\perp}.b_2|}\right)|\curl_{m_j}\beta|.\label{equiv1}
\ee
Moreover, the same inequalities hold for $\|\curl_{m_j}\beta\|_{\mathrm{lam}}$ relative to $\|\curl_{m_j}\beta\|$ when $\|\curl_{m_j}\beta\|$ is finite on $\Omega_1\subset\mathbb{R}^3$, and for $\|\beta\|_{1-\mathrm{lam}}$ relative to $\|\beta\|_1$ when $\beta\in L^1(\Omega_1)$.
\end{proposition}

\begin{proof}
All integrals in the following are taken with respect to $dydz$, at some fixed value of the $x$-coordinate. We may suppose w.l.o.g. that $b_1 = e_2$, which implies, using the notation above, that $s_2 = c_1 + (b_2.e_2)c_2$ and $s_3 = (b_2.e_3)c_2$. 

Then, using the triangle inequality for the total variation,
\bq
|\curl_{m_j}\beta| & = & \sup_{
\left|
\begin{array}{cc}
\phi_2 & \phi_3\\
\psi_2 & \psi_3
\end{array}
\right|\leq 1}
\int_{\Omega_1} (c_1 + (b_2.e_2)c_2)\dv\phi + (b_2.e_3)c_2\dv\psi\\
& \leq & \sup_{|\phi|\leq 1}\int_{\Omega_1} (c_1 + (b_2.e_2)c_2)\dv\phi + |b_2.e_3|\sup_{|\psi|\leq 1}\int_{\Omega_1} c_2\dv\psi \\
& \leq & \sup_{|\phi|\leq 1}\int_{\Omega_1} c_1\dv\phi + |b_2.e_2|\sup_{|\phi|\leq1}\int_{\Omega_1} c_2\dv\phi + |b_2.e_3|\sup_{|\psi|\leq 1}\int_{\Omega_1} c_2\dv\psi\\
& \leq & (|b_2.b_1| + |b_2.b_1^{\perp}|)|\curl_{m_j}\beta|_{\mathrm{lam}}.
\eq
Since $(|b_2.b_1| + |b_2.b_1^{\perp}|)$ is maximised when $b_1$ and $b_2$ meet at an angle of $\frac{\pi}{4}$, we also obtain the final inequality of (\ref{equiv2}).

Turning to (\ref{equiv1}), since
\be
c_2 = \frac{s_3}{b_2.e_3},\qquad c_1 = s_2 - \frac{(b_2.e_2)}{(b_2.e_3)}s_3,
\ee
and we have
\bq
|\curl_{m_j}\beta|_{\mathrm{lam}} & = & \sup_{|\phi|\leq1}\int_{\Omega_1} c_1\dv\phi + \sup_{|\psi|\leq 1}\int_{\Omega_1} c_2\dv\psi\\
& \leq &  \sup_{|\phi|\leq1}\int_{\Omega_1} s_2 \dv\phi + \left(\frac{1 + |b_2.e_2|}{|b_2.e_3|}\right)\sup_{|\psi\leq 1}\int_{\Omega_1} s_3\dv\psi\\
 & \leq &  \sup_{
\left|
\begin{array}{cc}
\phi_2 & \phi_3\\
\psi_2 & \psi_3
\end{array}
\right|\leq 1}
\sqrt{2}\left(\int_{\Omega_1} s_2\dv\phi + \left(\frac{1 + |b_2.e_2|}{|b_2.e_3|}\right)s_3\dv\psi
\right)\\
 & \leq & \sqrt{2}\left(\frac{1 + |b_1.b_2|}{|b_1^{\perp}.b_2|}\right)|\curl_{m_j}\beta|,
\eq
as required.

The remaining assertions of the proposition follow in a similar way.
\end{proof}

\begin{remark}
Inequality (\ref{equiv1}), with $b_1.b_2=0$, was implicitly used in \cite{ang} to get the energy upper bound for a particular shear experiment on a B2 single-crystal. 
\end{remark}

Next, the laminated curl also has the desirable property of being convex on a given slip patch.
\begin{proposition}
For fixed $\Omega_1\subset\mathbb{R}^3$, $m_j\in\mathcal{M}$, $t\in\mathbb{R}$ and $s\in BV((\Omega_1)_t,m_j^{\perp})$, the mapping $s\mapsto |\curl(s\otimes m_j)|_{\mathrm{lam}}(\Omega_1)_t$ is convex. The same assertion holds for the total laminated curl if $s:\Omega_1\mapsto m_j^{\perp}$ is such that $\|\curl_{m_j}\beta\|_{\mathrm{lam}}(\Omega_1)<\infty$, and also for the laminated hardening. 
\end{proposition}

\begin{proof}
Choose $m_j\in\mathbb{R}^3$ and $s, t\in BV(\Omega_x,m_j^{\perp})$. Set $r = \lambda s + (1-\lambda)t$. Then, by linearity, and with (hopefully) obvious notation,
\bq
|\curl (r\otimes m_j)|_{\mathrm{lam}} & = & \sup_{|\phi|\leq 1}\int_{\Omega_1} (\lambda c_1^s + (1-\lambda)c_1^t)\dv\phi + \sup_{|\psi|\leq1}\int_{\Omega_1} (\lambda c_2^s + (1-\lambda)c_2^t)\dv\psi\\
 & \leq & \lambda\left(\sup_{|\phi|\leq 1}\int_{\Omega_1} c_1^s\dv\phi + \sup_{|\psi|\leq 1}\int_{\Omega_1} c_2^s\dv\psi\right)\nonumber\\
 & & +~(1-\lambda)\left(\sup_{|\phi|\leq 1}\int_{\Omega_1} c_1^t\dv\phi + \sup_{|\psi|\leq 1}\int_{\Omega_1} c_2^t\dv\psi\right)\label{tri}\\
 & = & \lambda|\curl (s\otimes m_j)|_{\mathrm{lam}} + (1-\lambda)|\curl (t\otimes m_j)|_{\mathrm{lam}},
\eq
where we used the triangle inequality for the total variation to get (\ref{tri}).

The two last assertions of the proposition are, once again, trivial.
\end{proof}

\subsection{The displacement}
We now construct a zig-zag perturbation to the displacement $u$ which accommodates the laminated plastic distortion, for a given $\beta = s\otimes m$ and decomposition $s = \sum_{i=1}^2 c_ib_i$, with negligible addition of elastic energy. This perturbation, denoted by $\hat{u}_n$, is constructed as follows. On a given bi-layer of the $\beta_n$-laminate, we set $\hat{u}_n = 0$ on the bottom boundary, then on the $b_j$-slice we set
\be
\frac{\p\hat{u}_n}{\p x^1} = s - 2c_jb_j,\label{2_slice}
\ee
where the right-hand side of (\ref{2_slice}) is evaluated on the centre-plane of the bi-layer. By construction, this implies that $\hat{u}_n = 0$ on top of the bi-layer, and so this procedure can be carried out consistently on the whole slip patch $\Omega_1$. Also, by the smoothness of $\beta$, both $\hat{u}_n$ and $\nabla_{y,z}\hat{u}_n$ are $\mathcal{O}\left(\frac{1}{2^n}\right)$, uniformly on $\Omega_1$. Thus, by defining $u_n = u + \hat{u}_n$ and $\hat{\beta}_n = \beta_n - \beta$, we get
\be\label{disp}
\nabla\hat{u}_n  =  \hat{\beta}_n + \mathcal{O}\left(\frac{1}{2^n}\right),
\ee
and hence that the the linearised elastic energy of $(u_n,\beta_n)$ converges to that of $(u,\beta)$ as $n\ra\infty$, whereby $u_n\in H^1$ and $\beta_n\in L^{\infty}$ for all $n$.

The above applies to $(u,\beta)$ away from $\p\Omega_1$. Since we are assuming $\supp\beta$ to be Lipschitz and compactly included in $\Omega_1$, we can, near $\p\Omega_1$, make a small perturbation to $\hat{u}_n$ to ensure that the construction can be carried out on the whole of $\Omega$ while $u_n$ still satisfies the original Dirichlet condition. In particular, and in a completely standard way, for each $n$ and each laminate bi-layer, we clip off the bi-layer in the region $\Omega_1\setminus\supp\beta$ in order to make a thin cylinder which is compactly included in $\Omega_1$ and $o(1)$-distant from $\p\Omega_1$, which can be done if $n$ is large enough. Then, for each value of $x^1$ on $\Omega_1$ we linearly interpolate $\hat{u}_n$ down to zero over a distance of magnitude $o(1)$, such that $\hat{u}_n$ has a long tent shape on each bi-layer, and such that $\nabla\hat{u}_n=\mathcal{O}(1)$ outside $\supp\beta$. Clearly, the perturbation to the elastic energy due to this boundary correction is $o(1)$, and so 
the elastic energy of the laminate still converges to that of $(u,\beta)$ as $n\ra\infty$.

\subsection{The relaxed energy}
From the above considerations, we see that the correct expression for the relaxed energy, which makes sense for non-smooth $\beta$, and, for example, $u\in H^1$, is just 
\be
E_{\mathrm{rel}}(u,\beta) = \left\{
\begin{array}{ccl}
\int_{\Omega} | (\nabla u - \beta)_\mathrm{sym}|^2\dx + \sigma\|\curl_{\mathcal{M}}\beta\|_{\mathrm{lam}} + \tau\|\beta\|_{1-\mathrm{lam}} & : & \textbf{(RSC)}~\textrm{holds},\\
 & & \\
+\infty & : & \textrm{otherwise},
\label{rel_energy}
\end{array}
\right.
\ee
where $\|\curl_\mathcal{M}\beta\|_{\mathrm{lam}}$ and $\|\beta\|_{1-\mathrm{lam}}$ denote the laminated curl and laminated hardening, such that, by analogy with $\|\curl_\mathcal{M}\beta\|$ , $\|\curl_\mathcal{M}\beta\|_{\mathrm{lam}}$ is obtained by summing the laminated curls over all possible slip planes.

Given this, we can summarise the results of our lamination procedure as follows.
\begin{proposition}
\label{prop:rel_energy}
For any pair $(u,\beta)$, $u\in H^1$, $\beta$ smooth, on a Lipschitz domain $\Omega$, such that $\beta = s\otimes m_j$ ($m_j\in\mathcal{M}$) on $\Omega_j$ and $\supp s|_{\overline{\Omega}_j}\subset\subset(\Omega_j\cup(\p\Omega\cap\p\Omega_j))$ is Lipschitz, there exists a sequence $(u_n,\beta_n)$ such that $u_n\in H^1$, $u_n = u$ on $\p\Omega$, $\beta_n\in L^{\infty}$ satisfies the non-relaxed side condition \textbf{(SSC)}, and
\be
E(u_n,\beta_n) \ra E_{\mathrm{rel}}(u,\beta).
\ee
Moreover, the relaxed energy (\ref{rel_energy}) is convex on each $\Omega_j$, as a function of $s$.
\end{proposition}

We now wish to show that a larger class of (possibly non-smooth) pairs $(u,\beta)$ for which the relaxed energy (\ref{rel_energy}) is finite can be approximated by smooth functions which respect (RSC) and the boundary conditions, without increasing the curl.
 
\section{Smoothing}\label{sec:3}

In order to perform the lamination of Section \ref{sec:2}, we must first smooth (the Burgers components of) a given, relaxed $\beta$ without increasing the relaxed energy (\ref{rel_energy}), which is a little involved, due to the relaxed side condition on $\beta$, as noted in the introduction.

Our basic procedure on a given slip-patch $\Omega_1$ is to mollify the $c_i$ once, thereby perhaps violating the relaxed slip condition, but reducing the curl, then to cut off the resulting functions on a strictly interior approximation to $\Omega_1$, such that the curl barely increases, and finally to apply a standard mollifier, which re-establishes the relaxed slip condition if the convolution kernel is chosen fine enough, and which once more decreases the curl.

First of all, it is no loss of generality to assume boundedness of the $c_i$, which will be useful in the sequel. More explicitly:

\begin{proposition}\label{prop:bounded}
Suppose we have a domain $\Omega_1\subset\Omega$ and scalar functions $f$ and $g$ on $\Omega_1$ such that $gf\in L^p(\Omega_1)$ for some $p\in[1,\infty)$, $f=0$ on $\Omega\setminus\Omega_1$ and $\|\curl f\|(\Omega)<\infty$. Then $\exists$ a sequence $f^n\in L^{\infty}(\Omega)$ such that $f=0$ on $\Omega\setminus\Omega_1$, $gf^n\ra gf\in L^p(\Omega)$ as $n\ra\infty$, and $\|\curl f_n\|(\Omega)\leq\|\curl f\|(\Omega)$, $\forall n\in\mathbb{N}$.\label{bounded}
\end{proposition}

\begin{proof}
First define
\be
\bar{f}^n = 
\left\{
\begin{array}{ccc}
f & : & f<n\\
n   & : & f\geq n
\end{array}
\right..
\ee

Then $(g\bar{f}^n)^p\nearrow (gf)^p$ pointwise, and therefore monotone convergence implies $\|g\bar{f}^n\|_{L^p}\ra\|gf\|_{L^p}$ as $n\ra\infty$.

Let $\Omega_c = \Omega\cap\{x^1 = c\}$. By the co-area formula applied to $F_{t,c}^n = \{x^i\in\Omega_c:\bar{f}^n(x^i)>t\}$, we get
\bq
|D_{y,z}\bar{f}^n|(\Omega_c) & = & \int_{-\infty}^{\infty}\|\partial F_{t,c}^n\|(\Omega_c)~dt\\
 & = & \int_{-\infty}^n\|\partial F_{n,c}\|(\Omega_c)~dt\\
 & \leq & |D_{y,z}f|(\Omega_c).
\eq
Applying the same argument to
\be
f^n = 
\left\{
\begin{array}{ccc}
\bar{f}^n & : & \bar{f}^n>-n\\
-n   & : & \bar{f}^n\leq -n
\end{array}
\right.,
\ee
we arrive at
\be
|D_{y,z}f^n|(\Omega_c)\leq |D_{y,z}f|(\Omega_c),
\ee
and, by integrating in the $x^1$-direction, $\|\curl f^n\|(\Omega)\leq\|\curl f\|(\Omega)$. Clearly, we also have $gf_n\ra gf$, $(gf^n)^p\nearrow (gf)^p$, pointwise, and therefore by monotone convergence, $\|gf^n\|_{L^p}\ra\|gf\|_{L^p}$. Together with, say, Proposition 1.3.3 of \cite{amb}, we also get $gf_n\ra gf\in L^p$ as $n\ra\infty$, as required.
\end{proof}

We now show that a bounded $\beta$ supported on a slip-patch $\Omega_1\subset\subset\Omega$ can be mollified without enlarging the support, and such that the laminated curl essentially decreases, provided $\Omega_1$ satisfies an additional boundary-regularity condition. This condition states that $\Omega_1$ should have finite perimeter in $\Omega$, and that the perimeter should be equal to the area of the topological boundary: $\mathcal{H}^2(\p\Omega_1) = P(\Omega_1)$, which is equivalent to the statement $\mathcal{H}^2(\p\Omega_1\setminus\mathcal{F}\Omega_1)=0$, where $\mathcal{F}\Omega_1$ is the reduced boundary \cite{amb}, and which very roughly means that $\Omega_1$ has no cuts. The same condition was employed by Schmidt \cite{schmidt} in order to obtain strictly-interior approximations of BV domains and functions. Our approach to obtaining the desired cut-off is somewhat different to  (and less technical than) that of Schmidt, and only works in the scalar case, which is enough for us since the 
components of $\beta$ ($c_i$ in (\ref{c})) can be mollified separately when considering the relaxed energy. The advantage of our method is that we require less regularity on $\beta$ than does Schmidt, namely just $\beta\in L^p(\Omega)$, rather than $\beta\in BV\cap L^{\infty}(\Omega)$. The case where $\p\Omega_1$ intersects the sample boundary, $\p\Omega$, will be considered separately towards the end of this section.

\begin{proposition}
Let $\Omega_1\subset\subset\Omega\subset\mathbb{R}^3$ be a bounded domain with finite perimeter satisfying $\mathcal{H}^2(\p\Omega_1\setminus\mathcal{F}\Omega_1)=0$, and let $f\in L^p(\Omega)$ with $\mathrm{supp}f\subset\Omega_1$ be such that $\int|D_{y,z}f|~\mathrm{d}x^1<\infty$. Then there exists a family $f_{\ep}\in C_0^{\infty}(\Omega_1)$, $\ep>0$, such that
\be 
f_{\ep}\ra f\in L^p(\Omega_1)\quad\mathrm{as}\quad\ep\ra 0 \label{f1}
\ee
and
\be
\int_{\Omega_1}|D_{y,z}f_{\ep}|~\mathrm{d}x^1\leq\int_{\Omega_1}|D_{y,z}f|~\mathrm{d}x^1 + \ep \label{f2}.
\ee\label{f_prop}
\end{proposition}

\begin{proof}
By Proposition \ref{bounded}, we may assume w.l.o.g. that $f$ is bounded.

The first stage of the proof is to construct a strictly interior approximation to $\Omega_1$. To do this, we employ the Structure Theorem for sets of finite perimeter \cite{evgar}. Thus, since $\mathcal{F}\Omega_1$ is $\mathcal{H}^2$-almost all of $\p\Omega_1$, we have that $\forall\ep>0$ there exist $M(\ep)\in\mathbb{N}$ and compact, disjoint sets $K_k^{\ep}\subset\p\Omega_1$, $k=1,2,\ldots M$, such that each $K_k^{\ep}$ is a subset of a $C^1$-surface, $\nu_{\Omega_1}|_{K_k}$ is normal to $K_k^{\ep}$, and
\be
\p\Omega_1 = \left(\bigcup_{k=1}^{M(\ep)} K_k^{\ep}\right)\sqcup N,\quad\mathrm{where}\quad\mathcal{H}^2(N)\leq\ep,\label{N}
\ee
for some $\mathcal{H}^2$-measurable set $N\subset\p\Omega_1$, and where $\nu_{\Omega_1}|_{K_k}$ is the generalised outward unit normal to $\p\Omega_1$.

By compactness, the $K_k^{\ep}$ are pairwise separated by a distance $d(\ep)>0$, and we may assume that $K_k^{\ep} = \phi_k(\overline{U}_k^{\ep})$, where $U_k^{\ep}\subset U_k$ is an open coordinate patch in $\mathbb{R}^2$ (as is $U_k$), $\phi_k$ is a $C^1$-function and the overbar denotes topological closure. 

Now, with $\ep>0$ fixed, for each $k=1,2,\ldots M$, cover $U_k^{\ep}$ with an $\hep$-coordinate mesh, where $\hep>0$ is small, and remains to be determined. Let $S_{ij}^k$ be the open mesh squares, with centres denoted by $y^c_{ij}$, and define the slightly enlarged open squares $\widehat{S}_{ij}^k = (1 + \hep)S_{ij}^k$ by scaling about the $y^c_{ij}$ (here, $i$ and $j$ index the mesh lines in the $y^1$ and $y^2$ coordinate directions). Thus, the $\phi_k(\widehat{S}_{ij}^k\cap U^{\ep}_k)$ comprise a relatively open cover of $K_k^{\ep}$.

By compactness of the $U_k^{\ep}$ and the fact that $\phi_k$ is $C^1$, we have that for every $\delta>0$ there exists an $\hep(\delta, M)$ such that, for $y, \bar{y}\in U_k^{\ep}$,
\be
|y-\bar{y}|<2\hep(\delta, M)\quad\Rightarrow\quad |\nabla\phi_k(y)-\nabla\phi_k(\bar{y})| < \delta.\label{C1}
\ee
For fixed $\ep$ and $\delta$, we will henceforth choose the $\hep$-mesh on the $U^{\ep}_k$ so fine that (\ref{C1}) holds.

Now, denoting by $\mathrm{Tan}(S,x)$ the tangent plane of a manifold $S$ at the point $x$, for each $i, j, k$, we take the orthogonal projection of $\phi_k(\widehat{S}_{ij}^k)$ onto $\mathrm{Tan}(\widehat{K}_k^{\ep},x^c_{ij})$, where $x^c_{ij} = \phi_k(y^c_{ij})$, and call the result $P^k_{ij}$, then construct the cylinder
\be 
C^k_{ij} = P^k_{ij}\times l^k_{ij},\quad\mathrm{where}\quad l^k_{ij} = \{x^c_{ij}\pm t\hep\delta\nu_{x^c_{ij}} : t\in [0,1)\},
\ee
which contains $\phi_k(\widehat{S}_{ij}^k)$ if $\hep(\delta, M)$ is chosen as above, and is such that the outer face, defined by
\be
O^k_{ij} = P^k_{ij}\times\{x^c_{ij} + \hep\delta\nu_{x^c_{ij}}\},
\ee 
does not intersect $\Omega_1$, which can always be arranged if $\delta$ is small enough, by compactness. Also, if $\delta$ is small enough, then any given $C^k_{ij}$ intersects only those other $C^k_{pq}$ such that $S^k_{pq}$ is one of the (at most eight) immediate neighbours of $S_{ij}^k$.

Next, by the definition of Hausdorff measure, we may cover the $\mathcal{H}^2$-small set $N$ with countably many balls, $B_{\rho_i}(q_i)$ with $\rho_i\leq 1$, such that $\sum_{i=1}^{\infty}\rho_i^2\leq 2\ep$. Moreover, by compactness and (\ref{N}), we can cover $\p\Omega_1$ with the $C^k_{ij}$ and a finite sub-collection of the $B_{\rho_i}(q_i)$, say $B_{\rho_{m_i}}(q_{m_i})$ for $i = 1,2,\ldots M_1$, so that, defining 
\be
\Omega_{\ep} = \Omega_1\setminus\left(\left(\bigcup_{i = 1}^{M_1} B_{\rho_i}(q_i)\right)\bigcup\left(\bigcup_{ijk}C^k_{ij}\right)\right),
\ee
we see that 
\be
\Omega_{\ep}\subset\subset\Omega_1\quad \mathrm{and}\quad |\Omega_1\setminus\Omega_{\ep}| = \mathcal{O}(\ep + \hep\delta).
\ee
For each $C^k_{ij}$ we now subtract the shadow of its (thin) lateral portion, along with the shadow of the intersection with its (at most eight) neighbouring $C^r_{pq}$ and the $B_{\rho_l}(q_l)$, the shadow being taken with respect to the projection of $\nu_{x^c_{ij}}$ onto the $(x^2,x^3)$-plane. Denote by $\widehat{C}^k_{ij}$ the resulting disjoint sets, and observe that the sections $\widehat{C}^k_{ij}\cap\{x^1=\mathrm{const.}\}$ are necessarily rectangular, by construction. 

It follows from the $C^1$-property of the $\phi_k$ and (\ref{C1}) that the amount of projected area lost from $\p\Omega_{\ep}$ by removing shadows in this way is bounded by $C_{M(\ep)}(\hep(\delta, M(\ep)) + \delta) + 2\ep$ for some $C_{M(\ep)}$, and this can be made less than $3\ep$ for a given $\ep$, if $\delta$ and then $\hep$ are made sufficiently small. 

Next, suppose that we smooth $f$ with a Friedrichs mollifier so fine that the resulting function $\tilde{f}$ has support contained in $\Omega_1\cup(\cup_{ijk}C^k_{ij})\cup(\cup_i B_{\rho_{m_i}}(q_{m_i}))$. By a standard result on the mollification of measures, namely Thm.2.2(b) of \cite{amb}, this reduces the curl, and $\tilde{f}$ can be taken arbitrarily close to $f$ in $L^p$, by standard convolution theory. Also, by what we just said above, along with Young's inequality, if we cut off $\tilde{f}$ on $\p\Omega_{\ep}$, then the amount of curl on $\p\Omega_{\ep}$ which goes unaccounted for by neglecting shadows is bounded by $3\ep\|f\|_{\infty}$, and it remains to estimate the curl generated on the flat faces $\widehat{C}^k_{ij}\cap\p\Omega_{\ep}$.

Thus, on each $\widehat{C}^k_{ij}$ we consider the horizontal slices $x^1=t$, which, as we already noted, are rectangles. Since $\tilde{f}$ is a smooth scalar function and $\tilde{f} = 0$ on the outer edge of $\widehat{C}^k_{ij}\cap\{x^1=t\}$, we have
\be
\sum_{ijk}\int~\mathrm{d}t \int|D_{y,z}\tilde{f}|(\{x^1=t\}\cap\widehat{C}^k_{ij})\geq \sum_{ijk}\int~\mathrm{d}t\int_{\p\Omega_{\ep}\cap\widehat{C}^k_{ij}\cap\{x^1=t\}}|\tilde{f}|\mathrm{d}l\label{cut_curl},
\ee
where $\mathrm{d}l$ is the Euclidean line element.

The curl generated by cutting off $\tilde{f}$ on $\p\Omega_{\ep}$ is just the right-hand side of (\ref{cut_curl}) plus a contribution coming from the parts of $\p\Omega_{\ep}$ lost upon subtracting shadows, which we know is bounded by $3\ep\|f\|_{\infty}$ if $\hep$ and $\delta$ are chosen small enough. Thus, for appropriate $\hep$ and $\delta$, we may perform the cut-off and then (Friedrichs) mollify once more to obtain a smooth $f_{\ep}$ with smooth support which satisfies (\ref{f1}) and (\ref{f2}), modulo a relabelling of $\ep$.
\end{proof}

Proposition \ref{f_prop} can also be used to smooth $\beta$ on slip patches which intersect the sample boundary, $\p\Omega$. Here, we must take care not to introduce spurious curl when mollifying the slip near $\p\Omega$: a key tool in this connection is standard extension theory for BV-functions \cite{amb}.
\begin{proposition}
Suppose the sample domain $\Omega$ is Lipschitz, and that we have a slip patch $\Omega_1\subset\Omega$ satisfying the regularity condition $\mathcal{H}^2(\p\Omega_1\setminus\mathcal{F}\Omega_1)=0$. Let $f\in L^p(\Omega)$, with $\supp f\subset\Omega_1$, be such that $\|\curl f\|_{\Omega}<\infty$. Then there exists a family $f_{\ep}\in C^{\infty}(\Omega)$, $\ep>0$, with $f_{\ep}=0$ on $\Omega\setminus\Omega_1$, such that
\be
\mathrm{dist}\left(\supp f_{\ep}, \p\Omega_1\setminus\p\Omega\right)\geq\ep,
\ee
\be 
f_{\ep}\ra f \textrm{ in } L^p(\Omega_1)\quad\mathrm{as}\quad\ep\ra 0 \label{bdyf1}
\ee
and
\be
\int_{\Omega_1}|D_{y,z}f_{\ep}|~\mathrm{d}x^1\leq\int_{\Omega_1}|D_{y,z}f|~\mathrm{d}x^1 + \ep \label{bdyf2}.
\ee
\label{bdy}
\end{proposition}

\begin{proof}
Once more, we may assume that $f$ is bounded. Now focus attention on a slice $\Omega_t$ through $\Omega$. By Proposition 3.21 of \cite{amb}, for a.e. $t$ we can extend $f$ on $\Omega_t$, continuously in BV, to a function $\bar{f}\in BV({\mathbb{R}^2})$, with compact support, such that $|D_{y,z}\bar{f}|(\p\Omega_t) = 0$ (i.e. the sample boundary remains uncharged). Also, by inspecting the proof of this proposition, one can see that the BV operator norm of the extension map $f\mapsto\bar{f}$ is controlled by the Lipschitz constant of $\p\Omega_t$, which in turn is controlled by that of $\p\Omega$. In other words,
\be
\exists~L(\p\Omega)>0: \|\bar{f}\|_{BV(\mathbb{R}^2)}\leq L\|f\|_{BV(\Omega_t)},~\mathrm{for~a.e.~} t\in\R:\Omega_t\neq\emptyset.
\label{lip}
\ee

We also use $\bar{f}$ to denote the whole family of extensions over all $\Omega_t$.

Next, applying a 3-d Friedrichs mollifier to $\bar{f}$, and denoting the result by $\bar{f}_{\ep}$, we get
\bq
\iiint_{\Omega} |D_{y,z}\bar{f}_{\ep}| & \leq & \iiint_{\Omega^{\ep}} |D_{y,z}\bar{f}|\\
 & = & \int\mathrm{d}t\iint_{(\Omega^{\ep})_t} |D_{y,z}\bar{f}|~\mathrm{d}y\mathrm{d}z,
\eq
where $\Omega^{\ep}$ is the $\ep$-neighbourhood of $\Omega$, and the first inequality follows from Theorem 2.2(b) of \cite{amb}.

Furthermore, for each $t$, it follows that
\be
|D_{y,z}\bar{f}|((\Omega^{\ep})_t)  \searrow  |D_{y,z}\bar{f}|(\overline{\Omega}_t) = |D_{y,z}f|(\Omega_t),
\ee
as $\ep\ra 0^+$, since $D_{y,z}\bar{f}$ does not charge $\p\Omega$. 

Thus, we get, by monotone convergence and (\ref{lip}),
\be
\limsup_{\ep\ra 0^+}\iiint_{\Omega} |D_{y,z}(\bar{f}_{\ep})|~\mathrm{d}t\mathrm{d}y\mathrm{d}z\leq \int\mathrm{d}t~|D_{y,z}f|(\Omega_t),
\ee
as required. $L^p$-convergence of the mollified extensions is once again completely standard.

Once we have our smooth $\bar{f}_{\ep}$ on $\Omega$, we can find an $\Omega_1^{\ep}\subset\Omega_1$ such that $2\ep\leq\mathrm{dist}\left(\Omega_1^{\ep}, \p\Omega_1\setminus\p\Omega\right)\leq 3\ep$, and then cut off $\bar{f}_{\ep}$ on $\p\Omega_1^{\ep}\setminus\p\Omega$ for an $\Omega_1^{\ep}\subset\Omega_1$, by precisely the method of Proposition \ref{f_prop}, thereby essentially reducing the curl, and barely changing the $L^p$-norm.

Finally, we repeat the first part of the proof (extend across the domain boundary then smooth), but this time applied to the interior cut off of $\bar{f}_{\ep}$, rather than the initial $f$. This gives the required result if the mollifier is chosen fine enough.
\label{extension}
\end{proof}

\begin{remark}
\label{bdy_lam}
The smoothed extension constructed above can be used to laminate $\beta$ in the neighbourhood of the domain boundary, on a slip patch $\Omega_1$, in the case when $\p\Omega_1$ intersects $\p\Omega$.
\end{remark}

\section{The full relaxation statement}\label{sec:5}

By applying Proposition \ref{f_prop} or \ref{bdy} separately to each of the Burgers components of the plastic distortion $\beta$, for an optimal decomposition $s=\sum_ic_ib_i$, we can cut-off and laminate $\beta$ independently on neighbouring slip patches, and it follows that our lamination and smoothing results can be combined to arrive at the desired relaxation theorem.


\begin{theorem}
\label{thm:main}
Suppose we have a Lipschitz domain $\Omega\in\mathbb{R}^3$, and that $\overline{\Omega} = \bigcup_{i=1}^{\infty}\overline{\Omega_i}$, where the domains $\Omega_i$ are pairwise disjoint and satisfy the regularity condition $\mathcal{H}^2(\p\Omega_i\setminus\mathcal{F}\Omega_i)=0$, and suppose that we have $(u,\beta)$ on $\Omega$, such that $u\in H^1$ satisfies a Dirichlet condition on a Lipschitz subset of $\p\Omega$, $\beta = s\otimes m$ satisfies the relaxed-slip condition \textbf{(RSC)} with $m = m_i\in\mathcal{M}$ on each $\Omega_i$, and the relaxed energy (\ref{rel_energy}) is finite. Then, for each $\ep>0$, there exists a pair of test functions $(u_{\ep}, \beta_{\ep})$ satisfying the same Dirichlet condition and the single-slip condition \textbf{(SSC)}, such that $u_{\ep}\in H^1$, $\beta_{\ep}\in L^{\infty}$ and
\be
E(u_{\ep},\beta_{\ep}) \leq E_{\mathrm{rel}}(u,\beta) + \ep. 
\ee
\end{theorem}

\begin{remark}
Of course, what we would really like is to find the quasi-convex envelope of $E$ for general domains $\Omega$, and then prove the existence of minimisers. Note that, for the simple case where a slip plane connects free surfaces, one can achieve $E_{\textrm{rel}} = 0$ with a simple, relaxed shear (see, for example, the $L>2$ result from \cite{ang}) -- hence our theorem shows that the required envelope is in fact zero in this case. 
\end{remark}

\begin{remark}
Note that, while the class of $\beta$ we are considering here is rather large, we could, instead, allow those $\beta$ that can be approximated by smooth, compactly supported plastic distortions. As shown in Section \ref{sec:3} (which we think is of some independent interest), this class is no smaller than the one used in the above theorem -- it is, however, much harder to characterize. In the context of the previous remark, we believe that minimisers of the relaxed energy (with sufficiently smooth data) would satisfy the requirements of~\ref{thm:main}.
\end{remark}

\begin{remark}
\label{rem:non-infinite}
A $\int|\beta|^{\alpha}$ penalty term with $\alpha<1$ relaxes to zero, as one can see quite easily.
Thus, for a given single-slip one can laminate between the same kind of slip and zero slip, which strictly reduces the energy. Iterating this procedure to get ever finer laminates gives zero energy in the limit. It follows that if we replace $\int|\beta|$ with $\int|\beta|^{\alpha}$ then $E_{\mathrm{rel}}$ should just contain the elastic energy and the laminated curl.
\end{remark}

\section{Nonlinear elasto-plasticity}\label{sec:4}

We now show that the foregoing analysis of the geometrically linear problem can be carried over, with little additional work, to the nonlinear case. Here, we have in mind the usual multiplicative decomposition of the strain, $F = F^{\textrm{el}}(F^{\textrm{pl}})^{-1}$, so that the elastic deformation is given by
\be
F^{\mathrm{el}}(u,\beta) = (I + \nabla u)(I - \beta),
\ee
where $u$ is once again the material displacement, and, on a given slip patch, $\beta = s\otimes m$ (with $s\in m^{\perp}$) is the plastic slip, such that the family of possible slip systems has the same form as before. 

The non-relaxed, nonlinear energy to be considered is now
\be
E^{\textrm{(nl)}}(u,\beta) =
\left\{
\begin{array}{ccl}
\int_{\Omega} W_{\textrm{el}}(F^{\mathrm{el}})~\dx  + 
\sigma\|\curl_{\mathcal{M}}\beta\|_{\Omega} + \tau\int_{\Omega} |\beta|\dx & : & \textbf{(SSC)}~\textrm{holds},\\
 & & \\
+\infty & : &  \textrm{otherwise}, 
\label{nonlinear}
\end{array}
\right.
\ee
for some elastic-energy function $W_{\textrm{el}}$, and the relaxed, nonlinear energy is
\be
E^{\textrm{(nl)}}_{\mathrm{rel}}(u,\beta) = \left\{
\begin{array}{ccl}
\int_{\Omega} W_{\textrm{el}}(F^{\mathrm{el}})~\dx + \sigma\|\curl_{\mathcal{M}}\beta\|_{\mathrm{lam}} + \tau\|\beta\|_{1-\mathrm{lam}} & : & \textbf{(RSC)}~\textrm{holds},\\
 & & \\
+\infty & : & \textrm{otherwise},
\label{rel_nonlinear}
\end{array}
\right.
\ee

Here, we still use the linearised version of the density measure of the geometrically necessary dislocations. For an in-depth discussion on the controversy about which term is correct in the geometrically non-linear setting we refer to~\cite{Cermelli_2001} and~\cite{Reina_13}.

\begin{theorem}\label{thm:nonlinear}
Suppose that $W_{\textrm{el}}:\mathbb{M}^{3\times 3}\mapsto [0,\infty)$ is continuous and satisfies the $p$-growth condition
\be\label{W}
 -c_1 + c_2|F|^p \le W_{\textrm{el}}(F) \le C_1 + C_2|F|^p.
\ee
Suppose furthermore, as in Theorem~\ref{thm:main}, that we have a Lipschitz domain $\Omega\subset\mathbb{R}^3$, and that $\overline{\Omega} = \bigcup_{i=1}^{\infty}\overline{\Omega_i}$, where the domains $\Omega_i$ are pairwise disjoint and satisfy the regularity condition $\mathcal{H}^2(\p\Omega_i\setminus\mathcal{F}\Omega_i)=0$, and that we have $(u,\beta)$ on $\Omega$, such that $u\in W^{1,p}$ satisfies a Dirichlet condition on a Lipschitz subset of $\p\Omega$, $\beta = s\otimes m$ satisfies the relaxed-slip condition \textbf{(RSC)} with $m = m_i\in\mathcal{M}$ on each $\Omega_i$, and the relaxed energy (\ref{rel_nonlinear}) is finite. Then, for each $\ep>0$, there exists a pair of test functions $(u_{\ep}, \beta_{\ep})$ satisfying the same Dirichlet condition and the single-slip condition \textbf{(SSC)}, such that $u_{\ep}\in W^{1,p}$, $\beta_{\ep}\in L^{\infty}$ and
\be
E^\textrm{(nl)}(u_{\ep},\beta_{\ep}) \leq E^\textrm{(nl)}_{\mathrm{rel}}(u,\beta) + \ep. 
\ee
\end{theorem}

\begin{proof}
First, suppose we have a pair of relaxed test functions $(u,\beta)$ on a slip patch $\Omega_1$ of the type stipulated above, with finite energy, such that $\beta = s(x)\otimes m$ is smooth. Now laminate exactly as before to get a new function $\beta_n = \beta + \hat{\beta}_n$ and perturbations $\hat{u}_n$ and $\hat{\beta}_n$, such that (\ref{disp}) holds.

Put $y_n = y\circ\hat{y}_n$ $(\Rightarrow\nabla y_n = (\nabla y)(\nabla\hat{y}_n))$, such that $\hat{y}_n$ is defined by $\hat{y}_n = Ix + \hat{u}_n$. If we set $u_n = -Ix + y_n$, then, away from an $o(1)$-thin neighbourhood of $\p\Omega_1$, we get
\bq
F^{\mathrm{el}}(u_n,\beta_n) & = & (I + \nabla u_n)(I - \beta_n)\\
 & = & (\nabla y)(\nabla\hat{y}_n)(I - \beta - \hat{\beta}_n)\\
 & = & \nabla y(I + \nabla\hat{u}_n)(I - \beta - \hat{\beta}_n)\\
 & = & \nabla y(I - \beta) - (\nabla y)(\nabla\hat{u}_n)\beta + \mathcal{O}(1/2^n)\\
 & = & F^{\mathrm{el}}(u,\beta) + \mathcal{O}(1/2^n),
\eq
since $(\nabla\hat{u}_n)\beta = (\hat{\beta}_n + \mathcal{O}(1/2^n))\beta$ and $(\hat{\beta}_n)(\beta) = 0$. Also, we may interpolate $u_n$ to zero near $\p\Omega_1$ in the same way as before, so that $|\nabla u_n|$ is $\mathcal{O}(1)$ on a set of Lebesgue measure $o(1)$. Hence, by (\ref{W}), dominated convergence and continuity, we get, by taking an a.e.~converging subsequence, $\int_\Omega W_{\textrm{el}}(u_{n_k},\beta_{n_k})\ra\int_\Omega W_{\textrm{el}}(u,\beta)$ as $k\ra\infty$.

Next, we need to show that, given relaxed $(u,\beta)$ with finite energy and $(I + \nabla u)(I - \beta)\in L^p$, we can smooth $\beta$ without increasing the energy. By Proposition \ref{prop:bounded} and dominated convergence, $\beta$ may be assumed bounded, and then we know that our previous smoothing method gives a uniformly bounded, smooth family $\beta_{\ep}$ which takes care of the curl and hardening terms. It is also easy to see that the method gives $\beta_{\ep}\ra\beta\in L^p$ as $\epsilon\ra 0, q\in[1,\infty)$. Thus, possibly passing to a subsequence $\ep_k$ to get a.e. convergence, dominated convergence gives $\int_{\Omega} W_{\textrm{el}}(u,\beta_{\ep_k})\ra \int_{\Omega} W_{\textrm{el}}(u,\beta)$ as $\ep_k\ra 0$ since we have $I + \nabla u\in L^p$.

This completes the proof.

\end{proof}

\end{document}